\documentclass[11pt]{amsart}
\usepackage{amsmath,amssymb,amsthm,mathtools}
\usepackage[margin=1in]{geometry}
\usepackage{graphicx}
\usepackage{booktabs}
\usepackage[dvipsnames]{xcolor}
\usepackage[colorlinks=true,linkcolor=blue!60!black,citecolor=blue!60!black,urlcolor=blue!60!black]{hyperref}

\numberwithin{equation}{section}
\newtheorem{theorem}{Theorem}[section]
\newtheorem{proposition}[theorem]{Proposition}
\newtheorem{lemma}[theorem]{Lemma}
\newtheorem{corollary}[theorem]{Corollary}
\theoremstyle{definition}
\newtheorem{definition}[theorem]{Definition}
\newtheorem{assumption}[theorem]{Assumption}
\theoremstyle{remark}
\newtheorem{remark}[theorem]{Remark}

\makeatletter
\def\subsection{\@startsection{subsection}{2}%
  \z@{.7\linespacing\@plus\linespacing}{.5\linespacing}%
  {\normalfont\bfseries}}
\def\subsubsection{\@startsection{subsubsection}{3}%
  \z@{.5\linespacing\@plus.7\linespacing}{.3\linespacing}%
  {\normalfont\itshape}}
\makeatother

\newcommand{\pd}{\partial}
\newcommand{\E}{\mathbb{E}}
\newcommand{\R}{\mathbb{R}}
\newcommand{\C}{\mathbb{C}}
\newcommand{\G}{\mathcal{G}}
\newcommand{\K}{\mathcal{K}}
\newcommand{\F}{\mathcal{F}}

\newcommand{\tr}{\operatorname{tr}}

\title[Lie point symmetries of a dependent jump-diffusion risk model]{Lie point symmetry analysis of a dependent jump-diffusion risk model with exponential claim sizes}

\author{Aminu Ma'aruf Nass}
\address{Department of Actuarial Science, Federal University Dutse (FUD), P.M.B.~7156, Dutse, Jigawa State, Nigeria}
\email{aminu.nass@fud.edu.ng}

\author{Kassimu Mpungu}
\address{Department of Physical Sciences, Mountains of the Moon University (MMU), P.O.~Box 837, Fort Portal, Uganda}
\email{kmpungu@mmu.ac.ug}

\author{Rahmatullah Ibrahim Nuruddeen}
\address{Department of Mathematics, Federal University Dutse (FUD), P.O.~Box 7156, Dutse 720001, Jigawa State, Nigeria; Faculty of Education and Arts, Sohar University, Oman; and Gulf University for Science and Technology, Department of Mathematics and Natural Sciences, Kuwait}
\email{rahmatullah.n@fud.edu.ng}

\author{Yaqoob Al Rahbi}
\address{Faculty of Education and Arts, Sohar University, Oman}
\email{yrahbi@su.edu.om}


\begin{document}

\begin{abstract}
We carry out a Lie point symmetry analysis of the coupled integro-differential system satisfied by the conditional Gerber--Shiu functions of a dependent jump-diffusion risk model. In the model, claim sizes are exponential, the intensity of the next claim arrival depends on whether the current claim exceeds a random threshold, and both the volatility and the premium rate may depend on the surplus. We write the system as the restriction to the half-line of a far-field operator with prescribed exterior data, derive the complete set of determining equations, including the condition imposed by the nonlocal convolution term, and solve them. The resulting group classification holds for every surplus-dependent volatility and premium of the form considered and for every absolutely continuous or constant threshold distribution. A non-trivial symmetry beyond scaling and superposition exists if and only if the volatility and the premium rate are constant, and it is then exactly the translation in the surplus variable. We construct the optimal system of one-dimensional subalgebras. We show that the invariant solutions of the combined translation--scaling generators exist precisely when a $2\times2$ generalized Lundberg determinant vanishes, and we give this determinant for an arbitrary threshold distribution, with explicit forms for the constant and exponential thresholds. We prove existence and sign properties of the adjustment coefficient, derive its sensitivities to all model parameters, and establish Cram\'er--Lundberg-type asymptotics for the Gerber--Shiu functions, including the limit of the ratio between the two regimes. A numerical study shows that claim-size-driven dependence lowers the adjustment coefficient markedly below that of an independent model with the same mean claim intensity.
\end{abstract}

\maketitle
\markboth{A.~M.~NASS ET AL.}{LIE POINT SYMMETRIES OF A DEPENDENT JUMP-DIFFUSION RISK MODEL}

\section{Introduction}\label{sec:intro}

Dependent risk models, which allow correlation between claim sizes and inter-claim times, describe insurance portfolios more realistically than the classical independence assumptions \cite{Lundberg1903,Cramer1930,AsmussenAlbrecher2010}. Early extensions of the Cram\'er--Lundberg model added a diffusion perturbation to account for small fluctuations in income and expenses \cite{DufresneGerber1991,GerberLandry1998,TsaiWillmot2002}. Albrecher and Boxma \cite{AlbrecherBoxma2004} introduced a model in which the distribution of the next inter-claim time depends on whether the current claim exceeds a random threshold, and obtained exact Laplace transforms for the ruin probability. Related dependence structures were studied in \cite{AlbrecherTeugels2006,BoudreaultEtAl2006}; Badescu, Cheung and Landriault \cite{BadescuEtAl2009} used bivariate phase-type distributions; Cheung, Landriault, Willmot and Woo \cite{CheungEtAl2010} analysed structural properties of Gerber--Shiu functions \cite{GerberShiu1998} in dependent Sparre Andersen models; and diffusion-perturbed versions of the threshold model were treated in \cite{ZhangYang2011,Gao2016}. These works rely mainly on Laplace transforms, renewal arguments and martingale techniques. The invariance properties of the governing integro-differential systems have received little attention.

Lie point symmetry methods provide a systematic way to find the transformations that leave a differential equation invariant, the solutions invariant under them, and the reductions they induce \cite{Ovsiannikov1982,BlumanKumei1989,Olver1993,Ibragimov1994}. They have been extended to integro-differential equations \cite{GrigorievEtAl2010}. They have also been applied widely to stochastic differential equations \cite{GaetaQuintero1999,Unal2003,GaetaSpadaro2017,Nass2022,NassFredericks2017,Nass2019} and to nonlinear and fractional models \cite{AzadEtAl2010,KumarEtAl2022,YuFeng2025a,YuFeng2025b}. To the best of our knowledge, they have not been applied to the coupled Gerber--Shiu systems of dependent jump-diffusion risk models.

This paper carries out such an analysis. The governing system is linear but nonlocal: it couples two second-order differential operators with a convolution against the claim-size density, weighted by the threshold distribution. We treat the nonlocal term exactly, in the spirit of \cite{GrigorievEtAl2010}, and obtain the following results.
\begin{enumerate}
\item \emph{Far-field formulation.} The Gerber--Shiu system is the restriction to the positive half-line of a far-field operator acting on functions on $\R$, with the penalty function as exterior data (Proposition~\ref{prop:exterior}). The symmetry analysis is carried out for this operator.
\item \emph{Determining equations.} We derive the complete determining equations for fiber-preserving linear generators (Lemma~\ref{lem:determining}). They consist of four local conditions, from the coefficients of $m'''$, $m''$, $m'$ and $m$, and one kernel condition from the convolution term.
\item \emph{Group classification.} For surplus-dependent volatility $\sqrt{\sigma_1^2+\sigma_2^2u^2}$ and premium $c+au$, and for every absolutely continuous or constant threshold distribution, the symmetry algebra is spanned by the translation $\pd_u$, the scaling $Y_0$ and the superposition symmetries when $\sigma_2=a=0$, and by $Y_0$ and the superposition symmetries otherwise (Theorem~\ref{thm:classification}). The determining equations coming from the $m''$- and $m'$-coefficients alone admit further solutions; these are excluded by the $m$-coefficient and by the kernel condition (Remark~\ref{rem:partial}).
\item \emph{Invariant solutions and the generalized Lundberg equation.} The algebra $\langle\pd_u,Y_0\rangle$ is abelian; its optimal system is $\{\langle Y_0\rangle,\langle\pd_u-\gamma Y_0\rangle\}$. The solutions invariant under $\pd_u-\gamma Y_0$ exist precisely when $\det M(\gamma)=0$, where $M$ is a generalized Lundberg matrix given for an arbitrary threshold distribution (Theorem~\ref{thm:lundberg}). Its determinant has a reduced form that depends on volatility and premium only through the Laplace exponent $A\gamma^2-B\gamma$.
\item \emph{Adjustment coefficient and asymptotics.} We prove that an adjustment coefficient exists, determine the signs of the relevant quantities, derive closed-form sensitivities to every parameter, and establish Cram\'er--Lundberg-type asymptotics for the Gerber--Shiu functions with an explicit limiting ratio between the two regimes (Theorem~\ref{thm:asymptotics}).
\item \emph{The effect of dependence.} Numerically, claim-size-driven switching lowers the adjustment coefficient by up to about $8\%$ relative to an independent model with the same mean claim intensity, whereas random switching that ignores the claim size lowers it by about $2\%$ (Section~\ref{sec:numerics}).
\end{enumerate}

The paper is organized as follows. Section~\ref{sec:model} formulates the model and its Gerber--Shiu system. Section~\ref{sec:determining} derives the determining equations, and Section~\ref{sec:classification} solves them. Section~\ref{sec:invariant} constructs the optimal system, the invariant solutions and the generalized Lundberg equation. Section~\ref{sec:adjustment} studies the adjustment coefficient and the asymptotics of the Gerber--Shiu functions. Section~\ref{sec:numerics} gives a numerical illustration, and Section~\ref{sec:conclusion} discusses the results and concludes.

\section{Model formulation}\label{sec:model}

This section introduces the risk process and the threshold dependence, states the integro-differential system for the conditional Gerber--Shiu functions, and rewrites it in the far-field form used throughout the paper.

\subsection{Dependent jump-diffusion risk process}

We begin with the dynamics of the surplus. Let $W$ be a standard Brownian motion, $X_1,X_2,\dots$ i.i.d.\ claim sizes, exponentially distributed with parameter $\beta>0$, and $Y_1,Y_2,\dots$ i.i.d.\ thresholds with distribution function $G$ on $[0,\infty)$, all independent. The surplus process $U$ satisfies
\begin{equation}\label{eq:sde}
dU_t=B(U_t)\,dt+\sqrt{2A(U_t)}\,dW_t-dS_t,\qquad U_0=u\ge0,
\end{equation}
where $S_t=\sum_{k\le N_t}X_k$ and
\begin{equation}\label{eq:AB}
A(u)=\tfrac12\big(\sigma_1^2+\sigma_2^2u^2\big),\qquad B(u)=c+au,
\end{equation}
with $\sigma_1>0$, $c>0$ and $\sigma_2^2,a\ge0$. The term $\sigma_2^2u^2$ models volatility that grows with the surplus, as for a surplus invested in a risky asset, and $a$ models premium or interest income proportional to the surplus. The claim arrivals are dependent in the sense of Albrecher and Boxma \cite{AlbrecherBoxma2004}: if $X_k>Y_k$, the next inter-claim time is exponentially distributed with rate $\lambda_1$ (high-risk regime), and otherwise with rate $\lambda_2$ (low-risk regime), where $\lambda_1\neq\lambda_2$. We write $J_t\in\{1,2\}$ for the regime of the inter-claim interval in progress at time $t$, so that $(U,J)$ is a Markov process. After a claim of size $y$ the new regime is $1$ with probability $G(y)$.

\subsection{Conditional Gerber--Shiu functions}

Next we define the quantities of interest and state their equations. Let $\tau=\inf\{t\ge0:U_t<0\}$, with ruin also occurring when the diffusion carries the surplus to $0$. Let $\delta>0$ be the force of interest and $w:[0,\infty)\to\R$ a bounded penalty applied to the deficit $|U_\tau|$. Define
\[
m_j(u)=\E\big[e^{-\delta\tau}w(|U_\tau|)\mathbf 1_{\{\tau<\infty\}}\,\big|\,U_0=u,\ J_0=j\big],\qquad j=1,2,
\]
and $\hat w=\int_0^\infty\beta e^{-\beta z}w(z)\,dz$. If the initial regime is drawn as after a claim of size $X$, the unconditional Gerber--Shiu function is $\pi m_1+(1-\pi)m_2$ with $\pi=\mathbb{P}(Y<X)$.

\begin{proposition}\label{prop:gs}
Assume $m_1,m_2\in C^2(0,\infty)\cap C[0,\infty)$. Then, for $u>0$ and $j=1,2$,
\begin{multline}\label{eq:gs}
A(u)m_j''(u)+B(u)m_j'(u)-(\lambda_j+\delta)m_j(u)\\
+\lambda_j\int_0^u\beta e^{-\beta y}\big[G(y)m_1(u-y)+(1-G(y))m_2(u-y)\big]dy+\lambda_j\hat w\,e^{-\beta u}=0,
\end{multline}
with $m_j(0)=w(0)$ and $m_j(u)\to0$ as $u\to\infty$.
\end{proposition}

\begin{proof}
The generator of $(U,J)$ acts on $f=(f_1,f_2)$ by
\[
(\mathcal Af)_j(u)=Af_j''+Bf_j'+\lambda_j\Big(\int_0^\infty\beta e^{-\beta y}\big[G(y)f_1(u-y)+(1-G(y))f_2(u-y)\big]dy-f_j(u)\Big).
\]
The first two terms come from It\^o's formula for \eqref{eq:sde} between claims. The last term comes from claims arriving at rate $\lambda_j$ and switching the regime to $1$ with probability $G(y)$. Setting $m_j(x)=w(-x)$ for $x<0$, Dynkin's formula for $e^{-\delta t}m_{J_t}(U_t)$ up to $\tau$ and the strong Markov property give $(\mathcal Am)_j=\delta m_j$ on $(0,\infty)$; see \cite{GerberLandry1998,TsaiWillmot2002,ZhangYang2011} for the constant-coefficient case. For $y>u$ the integrand equals $\beta e^{-\beta y}w(y-u)$ in either regime, and integrating it gives $\hat we^{-\beta u}$. This yields \eqref{eq:gs}. Since $A(0)>0$, ruin by oscillation is immediate at $u=0$, so $m_j(0)=w(0)$.

For the limit, fix $t>0$. Then $|m_j(u)|\le\|w\|_\infty\big(\mathbb{P}_u(\tau\le t)+e^{-\delta t}\big)$. Let $f(x)=(1+x)^{-1}$ for $x\ge0$ and $f(x)=1$ for $x<0$. Then $Bf'+Af''\le K_1f$ on $[0,\infty)$, because $A$ and $B$ grow at most quadratically and linearly. Moreover $\E[f(x-X)]\le f(x/2)+e^{-\beta x/2}\le K_2f(x)$. Hence $\E_u[f(U_{t\wedge\tau})]\le e^{Kt}f(u)$, and since $f(U_\tau)=1$, $\mathbb{P}_u(\tau\le t)\le e^{Kt}/(1+u)$. Letting $u\to\infty$ and then $t\to\infty$ gives $m_j(u)\to0$.
\end{proof}

\begin{remark}
Only the two conditions $m_j(0)=w(0)$ and $m_j(\infty)=0$ are imposed. The value $m_j'(0)$ is part of the solution and is not prescribed.
\end{remark}

\subsection{Far-field formulation}

Finally, we rewrite the system in a form suited to symmetry analysis. For $m=(m_1,m_2)$ defined on $\R$, set
\begin{align}
(\K m)(u)&=\int_0^\infty\beta e^{-\beta y}\big[G(y)m_1(u-y)+(1-G(y))m_2(u-y)\big]dy,\label{eq:K}\\
(\G m)_j(u)&=A(u)m_j''(u)+B(u)m_j'(u)-(\lambda_j+\delta)m_j(u)+\lambda_j(\K m)(u),\qquad j=1,2.\label{eq:G}
\end{align}
These operators act on the class
\[
\F=\Big\{m\in C^\infty(\R;\C^2):\ \forall k\ \exists\,C_k>0,\ \theta_k<\beta\ \text{with}\ |m^{(k)}(x)|\le C_ke^{-\theta_kx}\ \text{for }x\le0\Big\}.
\]
For $m\in\F$ the integral in \eqref{eq:K} converges absolutely together with its derivatives, and $(\K m)'=\K(m')$.

\begin{proposition}\label{prop:exterior}
Extend $m_j$ to $x<0$ by $m_j(x)=w(-x)$. Then \eqref{eq:gs} holds if and only if $(\G m)(u)=0$ for all $u>0$.
\end{proposition}

\begin{proof}
Split \eqref{eq:K} at $y=u$. The part over $(0,u)$ is the convolution in \eqref{eq:gs}. On $(u,\infty)$ both $m_1(u-y)$ and $m_2(u-y)$ equal $w(y-u)$, so that part equals $\hat we^{-\beta u}$.
\end{proof}

The far-field operator $\G$ captures the bulk dynamics of the surplus. The exterior data break translation invariance trivially, as any boundary does, and contribute only terms of order $e^{-\beta u}$. Accordingly, we study the symmetries of $\G$.

\section{Determining equations}\label{sec:determining}

In this section we define the symmetries considered and derive the complete system of determining equations, including the condition coming from the nonlocal convolution term.

\subsection{Generators and the invariance condition}

We first fix the class of generators. Since \eqref{eq:G} is linear, we consider fiber-preserving infinitesimal generators that are linear in the dependent variables,
\begin{equation}\label{eq:X}
X=\xi(u)\,\pd_u+\sum_{i=1}^2\big(\alpha_i(u)m_i+h_i(u)\big)\pd_{m_i},
\end{equation}
with $\xi,\alpha_1,\alpha_2\in C^\infty(\R)$ of at most polynomial growth and $h=(h_1,h_2)\in\F$. The characteristic of $X$ is the vector $Q_X[m]=(\alpha_im_i-\xi m_i'+h_i)_{i=1,2}$. By the classical theory of linear equations \cite{BlumanKumei1989,Olver1993}, extended to integro-differential equations in \cite{GrigorievEtAl2010}, the infinitesimal invariance of a linear equation $\G m=0$ under $X$ means that the linearized operator applied to the characteristic is a combination of the equation and its derivative. This leads to the following definition. It is the notion of a first-order symmetry operator in the sense of Miller \cite{Miller1977}, and it is the condition used by the method of \cite{GrigorievEtAl2010} for linear equations.

\begin{definition}\label{def:sym}
The generator \eqref{eq:X} is \emph{admitted} by $\G m=0$ if $\G h=0$ and there exist smooth functions $\psi_1,\psi_2$ such that, for every $m\in\F$ and $j=1,2$,
\begin{equation}\label{eq:inv}
\big(\G[\alpha m-\xi m']\big)_j=-\xi\,(\G m)_j'+\psi_j\,(\G m)_j,
\end{equation}
where $(\alpha m)_i=\alpha_im_i$.
\end{definition}

If $X$ is admitted and $m$ solves $\G m=0$, then $Q_X[m]$ is again a solution. Hence the one-parameter group generated by $X$ maps solutions to solutions. The generators $h\cdot\pd_m$ with $\G h=0$ are the superposition symmetries, and the scaling $Y_0=m_1\pd_{m_1}+m_2\pd_{m_2}$ ($\xi=0$, $\alpha_1=\alpha_2=1$, $\psi_j=1$) is always admitted.

\subsection{The local part}

Next we compute the action on the differential part of $\G$. Write $L_j[f]=Af''+Bf'-c_jf$ with $c_j=\lambda_j+\delta$, and let $f=m_j$. Direct differentiation gives
\begin{multline*}
L_j[\alpha_jf-\xi f']=-A\xi f'''+\big(A\alpha_j-2A\xi'-B\xi\big)f''+\big(2A\alpha_j'-A\xi''+B\alpha_j-B\xi'+c_j\xi\big)f'\\
+\big(A\alpha_j''+B\alpha_j'-c_j\alpha_j\big)f,
\end{multline*}
and
\[
-\xi\,(L_jf)'+\psi_jL_jf=-A\xi f'''+\big(\psi_jA-\xi A'-\xi B\big)f''+\big(\psi_jB-\xi B'+c_j\xi\big)f'-\psi_jc_jf .
\]
The quantities $f'''(u),f''(u),f'(u),f(u)$ can be prescribed arbitrarily at any point, independently of the nonlocal term. Hence the local parts of \eqref{eq:inv} agree if and only if the coefficients agree.

\subsection{The nonlocal part}

We now treat the convolution term, which is the new ingredient. Write $\K m=\int_0^\infty\big[k_1(y)m_1(u-y)+k_2(y)m_2(u-y)\big]dy$ with $k_1=\beta e^{-\beta y}G$ and $k_2=\beta e^{-\beta y}(1-G)$. The convolution parts of \eqref{eq:inv} agree if and only if, for all $m\in\F$,
\[
\K[\alpha m]-\K[\xi m']+\xi(\K m)'-\psi_j\K m=0 .
\]
Since $(\K m)'=\K(m')$, the $i$-th part of the difference $\K[\xi m']-\xi(\K m)'$ is
\[
\int_0^\infty k_i(y)\big(\xi(u-y)-\xi(u)\big)m_i'(u-y)\,dy=\int_0^\infty\frac{\pd}{\pd y}\Big[k_i(y)\big(\xi(u-y)-\xi(u)\big)\Big]m_i(u-y)\,dy,
\]
by integration by parts, using $m_i'(u-y)=-\pd_ym_i(u-y)$ and the vanishing of the boundary terms at $y=0$ (where $\xi(u-y)-\xi(u)=0$) and at $y=\infty$ (by the bounds defining $\F$). The convolution condition therefore reads $\sum_i\int_0^\infty\Theta_{ij}(u,y)m_i(u-y)\,dy=0$, with
\[
\Theta_{ij}(u,y)=k_i(y)\big(\alpha_i(u-y)-\psi_j(u)\big)-\frac{\pd}{\pd y}\Big[k_i(y)\big(\xi(u-y)-\xi(u)\big)\Big].
\]
Since $m_1$ and $m_2$ restricted to $(-\infty,u)$ are arbitrary, this holds for all $m$ if and only if $\Theta_{ij}(u,\cdot)=0$ almost everywhere.

\subsection{The determining equations}

Collecting the local and nonlocal conditions gives the complete system.

\begin{lemma}\label{lem:determining}
The generator \eqref{eq:X} is admitted if and only if $\G h=0$ and, for $j=1,2$ and all $u\in\R$,
\begin{align}
&A\,(\alpha_j-\psi_j)=2A\xi'-\xi A', \tag{D1}\label{eq:D1}\\
&A\xi''-2A\alpha_j'+B\,(\xi'-\alpha_j+\psi_j)-\xi B'=0, \tag{D2}\label{eq:D2}\\
&A\alpha_j''+B\alpha_j'-c_j\,(\alpha_j-\psi_j)=0, \tag{D3}\label{eq:D3}\\
&\frac{\pd}{\pd y}\Big[k_i(y)\big(\xi(u-y)-\xi(u)\big)\Big]=k_i(y)\big(\alpha_i(u-y)-\psi_j(u)\big)\quad\text{for a.e. } y>0,\ i=1,2. \tag{D4}\label{eq:D4}
\end{align}
\end{lemma}

\begin{proof}
The coefficient of $f'''$ is $-A\xi$ on both sides. Equating the coefficients of $f''$, $f'$ and $f$ in the two displays of the local part gives \eqref{eq:D1}--\eqref{eq:D3}. Explicitly, the $f''$-coefficient gives $A\alpha_j-2A\xi'-B\xi=\psi_jA-\xi A'-\xi B$, which is \eqref{eq:D1}. The $f'$-coefficient gives $2A\alpha_j'-A\xi''+B\alpha_j-B\xi'=\psi_jB-\xi B'$, which is \eqref{eq:D2}. The $f$-coefficient gives $A\alpha_j''+B\alpha_j'-c_j\alpha_j=-\psi_jc_j$, which is \eqref{eq:D3}. The nonlocal part gives \eqref{eq:D4}.
\end{proof}

\begin{remark}\label{rem:partial}
Equations \eqref{eq:D1}--\eqref{eq:D2} alone, with $\psi_j\equiv0$ and $\alpha_1=\alpha_2$, reduce to the second-order equation
\[
3A^2\xi''+A(B-2A')\xi'-\big[A'(B-2A')+A(2A''-B')\big]\xi=0
\]
for $\xi$. For constant $A=\tfrac12\sigma_1^2$ and $B=c$, this has the solutions $\xi\in\operatorname{span}\{1,e^{-\gamma u}\}$ with $\gamma=2c/(3\sigma_1^2)$, and for $\sigma_2=0$, $a\neq0$ the Gaussian-type solution $\xi=\exp\{-(au+2c)u/(3\sigma_1^2)\}$. These candidates are not symmetries. For $\xi=e^{-\gamma u}$ and $\alpha=-2\gamma e^{-\gamma u}$, the left side of \eqref{eq:D3} equals $2\gamma e^{-\gamma u}\big[4c^2/(9\sigma_1^2)+\lambda_j+\delta\big]\neq0$, and \eqref{eq:D4} fails as well (Theorem~\ref{thm:classification}). All four groups of determining equations must therefore be imposed.
\end{remark}

\section{Group classification}\label{sec:classification}

We now solve the determining equations completely. The dependence between claims and inter-claim times ($\lambda_1\neq\lambda_2$) and the structure of the convolution kernel together make the classification rigid.

\subsection{Hypothesis on the threshold distribution}

We first state the hypothesis on $G$ under which the classification is proved.

\begin{assumption}\label{ass:G}
Either (a) $G$ is absolutely continuous on $[0,\infty)$ and not identically $0$ or $1$ on $(0,\infty)$, or (b) $G\equiv g_0$ on $(0,\infty)$ for a constant $g_0\in(0,1)$ (the next regime is chosen at random, independently of the claim size).
\end{assumption}

Case (a) includes the exponential threshold $G(y)=1-e^{-\kappa y}$. In case (a) the set $E=\{y>0:0<G(y)<1,\ G'(y)>0\}$ has positive Lebesgue measure.

\subsection{The classification theorem}

The main result of this section is the following.

\begin{theorem}\label{thm:classification}
Let $\lambda_1\neq\lambda_2$, let $A,B$ be given by \eqref{eq:AB}, and let Assumption~\ref{ass:G} hold. Then the generator \eqref{eq:X} is admitted by $\G m=0$ if and only if $\G h=0$, $\alpha_1=\alpha_2$ is constant, and
\begin{enumerate}
\item[(i)] $\xi$ is constant, if $\sigma_2=0$ and $a=0$;
\item[(ii)] $\xi\equiv0$, if $\sigma_2\neq0$ or $a\neq0$.
\end{enumerate}
Consequently, the symmetry algebra is
\[
\operatorname{span}\{\pd_u,\,Y_0\}\oplus\{h\cdot\pd_m:\G h=0\}\quad\text{if }\sigma_2=a=0,\qquad
\operatorname{span}\{Y_0\}\oplus\{h\cdot\pd_m:\G h=0\}\quad\text{otherwise.}
\]
\end{theorem}

\begin{proof}
\emph{Step 1: the multipliers coincide.} By Assumption~\ref{ass:G}, neither $k_1$ nor $k_2$ vanishes identically. Comparing \eqref{eq:D4} for $j=1$ and $j=2$ with the same $i$ gives $k_i(y)\big(\psi_1(u)-\psi_2(u)\big)=0$, hence $\psi_1=\psi_2=:\psi$.

\emph{Step 2: the scalings coincide.} By \eqref{eq:D1}, $\alpha_j-\psi=2\xi'-\xi A'/A$ for $j=1,2$, so $\alpha_1=\alpha_2=:\alpha$.

\emph{Step 3: $\alpha=\psi$.} Subtracting \eqref{eq:D3} for $j=1$ and $j=2$ gives $(c_1-c_2)(\alpha-\psi)=0$. Since $c_1-c_2=\lambda_1-\lambda_2\neq0$, $\alpha=\psi$. Then \eqref{eq:D3} reduces to $A\alpha''+B\alpha'=0$, and \eqref{eq:D1} to
\begin{equation}\label{eq:xiA}
2A\xi'=\xi A'.
\end{equation}

\emph{Step 4: $\xi$ is constant.} Consider first case (a) of Assumption~\ref{ass:G}. On the set $E$ both $k_1$ and $k_2$ are positive and differentiable, with logarithmic derivatives $\ell_1=-\beta+G'/G$ and $\ell_2=-\beta-G'/(1-G)$. Dividing \eqref{eq:D4} by $k_i$ gives, for $y\in E$ and all $u$,
\[
\ell_i(y)\big(\xi(u-y)-\xi(u)\big)-\xi'(u-y)=\alpha(u-y)-\psi(u),\qquad i=1,2 .
\]
The right side does not depend on $i$. Subtracting the two equations gives
\[
\frac{G'(y)}{G(y)(1-G(y))}\big(\xi(u-y)-\xi(u)\big)=0,\qquad y\in E .
\]
Hence $\xi(u-y)=\xi(u)$ for all $u\in\R$ and all $y\in E$. So every element of $E$ is a period of the continuous function $\xi$. By Steinhaus' theorem, $E-E$ contains an interval around $0$, so $\xi$ has arbitrarily small periods and is therefore constant.

In case (b), $k_i=\beta g_ie^{-\beta y}$ with $g_1=g_0$ and $g_2=1-g_0$, so $\ell_1=\ell_2=-\beta$. With $s=u-y$, \eqref{eq:D4} becomes
\[
\xi'(s)+\beta\xi(s)+\alpha(s)=\beta\xi(u)+\psi(u)\qquad\text{for all }s<u .
\]
The left side depends only on $s$ and the right side only on $u$, so both equal a constant $\kappa_0$. Thus $\psi=\kappa_0-\beta\xi$ and $\alpha=\kappa_0-\beta\xi-\xi'$. Since $\alpha=\psi$ by Step~3, $\xi'=0$, and $\xi$ is constant.

\emph{Step 5: $\alpha$ is constant.} With $\xi$ constant, the left side of \eqref{eq:D4} vanishes, so $\alpha(u-y)=\psi(u)$ for almost every $y$ with $k_i(y)\neq0$. Hence $\alpha$ is constant, and $\psi=\alpha$.

\emph{Step 6: the role of the coefficients.} With $\xi$ and $\alpha$ constant, \eqref{eq:xiA} reads $\xi A'=0$, and \eqref{eq:D2} reads $-\xi B'=0$, while \eqref{eq:D3} holds trivially. If $\sigma_2=a=0$, then $A'=B'=0$ and every constant $\xi$ is allowed, which is case (i). If $\sigma_2\neq0$ then $A'\not\equiv0$, and if $a\neq0$ then $B'\not\equiv0$; in both cases $\xi=0$, which is case (ii).

\emph{Converse.} If $\xi$ and $\alpha$ are constants and $\sigma_2=a=0$, then \eqref{eq:D1}--\eqref{eq:D4} hold with $\psi=\alpha$. If $\xi=0$ and $\alpha$ is constant, they hold for any $A,B$.
\end{proof}

\subsection{Consequences}

We record two consequences of the theorem.

\begin{corollary}\label{cor:breaking}
If the volatility or the premium depends on the surplus ($\sigma_2\neq0$ or $a\neq0$), the Gerber--Shiu system admits no Lie point symmetry beyond scaling and superposition. In particular, it admits no non-trivial similarity reduction.
\end{corollary}

\begin{remark}\label{rem:indep}
The hypothesis $\lambda_1\neq\lambda_2$ is used only in Step~3. In the independent case $\lambda_1=\lambda_2$ the two components satisfy the same scalar equation, and \eqref{eq:D3} no longer forces $\alpha=\psi$. Steps~4--6 then need to be redone for the scalar equation. This is a separate classification, which we do not pursue; the dependence between claims and inter-claim times is precisely what makes the classification above rigid.
\end{remark}

\section{Optimal system, invariant solutions and the generalized Lundberg equation}\label{sec:invariant}

Throughout this section $\sigma_2=a=0$, so $A\equiv A_0=\tfrac12\sigma_1^2$ and $B\equiv B_0=c$, and $\G_0$ denotes the corresponding operator. We construct the optimal system of one-dimensional subalgebras, determine the invariant solutions, and show that they are governed by a generalized Lundberg equation.

\subsection{Optimal system}

We begin with the algebra itself. By Theorem~\ref{thm:classification}, the finite-dimensional part of the symmetry algebra is spanned by
\[
X_1=\pd_u,\qquad Y_0=m_1\pd_{m_1}+m_2\pd_{m_2}.
\]
Their commutator is $[X_1,Y_0]=0$, so the algebra is abelian and its adjoint action is trivial. Hence two one-dimensional subalgebras are equivalent if and only if they coincide.

\begin{theorem}\label{thm:optimal}
An optimal system of one-dimensional subalgebras of $\operatorname{span}\{X_1,Y_0\}$ is
\[
\big\{\langle Y_0\rangle\big\}\cup\big\{\langle X_1-\gamma Y_0\rangle:\gamma\in\R\big\}.
\]
If complex invariant solutions are admitted, $\gamma$ ranges over $\C$.
\end{theorem}

\begin{proof}
Every non-zero element is $aX_1+bY_0$. If $a=0$ it spans $\langle Y_0\rangle$. Otherwise, dividing by $a$ gives $X_1-\gamma Y_0$ with $\gamma=-b/a$. Since the adjoint action is trivial, distinct subalgebras are inequivalent.
\end{proof}

\subsection{Invariant solutions}

Next we determine the solutions invariant under each subalgebra. A solution is invariant under $X$ if its characteristic vanishes. For $Y_0$ the characteristic is $m$, so the only invariant solution is $m=0$. For $X_\gamma=X_1-\gamma Y_0$ the characteristic is $-m'-\gamma m$, so the invariant functions are $m(u)=e^{-\gamma u}k$ with $k\in\C^2$ constant. They belong to $\F$ precisely when $\operatorname{Re}\gamma<\beta$. To substitute them into the equation, introduce, for $\operatorname{Re}\gamma<\beta$,
\begin{equation}\label{eq:pg}
p(\gamma)=\frac{\beta}{\beta-\gamma},\qquad g(\gamma)=\beta\int_0^\infty e^{-(\beta-\gamma)y}G(y)\,dy,\qquad \Phi(\gamma)=A_0\gamma^2-B_0\gamma,
\end{equation}
and the \emph{generalized Lundberg matrix}
\begin{equation}\label{eq:M}
M(\gamma)=\begin{pmatrix}
\Phi(\gamma)-\lambda_1-\delta+\lambda_1g(\gamma) & \lambda_1\big(p(\gamma)-g(\gamma)\big)\\[2pt]
\lambda_2g(\gamma) & \Phi(\gamma)-\lambda_2-\delta+\lambda_2\big(p(\gamma)-g(\gamma)\big)
\end{pmatrix}.
\end{equation}
Write $d(\gamma)=\det M(\gamma)$ and $T(\gamma)=\tr M(\gamma)$.

\begin{theorem}\label{thm:lundberg}
Let $\operatorname{Re}\gamma<\beta$. Then $\G_0[e^{-\gamma u}k]=e^{-\gamma u}M(\gamma)k$. Consequently, the far-field system admits a non-zero $X_\gamma$-invariant solution if and only if
\begin{equation}\label{eq:lundberg}
d(\gamma)=\det M(\gamma)=0,
\end{equation}
and the invariant solutions are $e^{-\gamma u}k$ with $k\in\ker M(\gamma)$.
\end{theorem}

\begin{proof}
For $m=e^{-\gamma u}k$ we have $m_j'=-\gamma m_j$ and $m_j''=\gamma^2m_j$, so the local part of row $j$ equals $(\Phi(\gamma)-\lambda_j-\delta)e^{-\gamma u}k_j$. In the convolution, $m_i(u-y)=e^{-\gamma u}e^{\gamma y}k_i$, so
\[
\K m(u)=e^{-\gamma u}\Big[k_1\beta\!\int_0^\infty\!e^{-(\beta-\gamma)y}G(y)\,dy+k_2\beta\!\int_0^\infty\!e^{-(\beta-\gamma)y}\big(1-G(y)\big)dy\Big]=e^{-\gamma u}\big[g\,k_1+(p-g)\,k_2\big].
\]
Adding $\lambda_j\K m$ to the local part gives row $j$ of $M(\gamma)k$, times $e^{-\gamma u}$. A non-zero $k$ with $M(\gamma)k=0$ exists if and only if $d(\gamma)=0$.
\end{proof}

Equation \eqref{eq:lundberg} is the \emph{generalized coupled Lundberg equation} of the model. Theorem~\ref{thm:lundberg} derives it entirely from the invariance of the far-field system under $X_\gamma$.

\subsection{Structure of the Lundberg determinant}

The determinant has a simple structure that makes it easy to analyse.

\begin{lemma}\label{lem:reduced}
With $a_j(\gamma)=\Phi(\gamma)-\lambda_j-\delta$,
\begin{equation}\label{eq:reduced}
d=a_1a_2+\lambda_1\,g\,a_2+\lambda_2\,(p-g)\,a_1 .
\end{equation}
In particular, $d$ depends on the volatility, the premium and the discount rate only through $\Phi-\delta$, and $\pd d/\pd\Phi=-\pd d/\pd\delta=T$.
\end{lemma}

\begin{proof}
$M=D+(\lambda_1,\lambda_2)^{\mathsf T}(g,\,p-g)$ with $D=\operatorname{diag}(a_1,a_2)$. By the matrix determinant lemma, $\det(D+xy^{\mathsf T})=\det D+y^{\mathsf T}\operatorname{diag}(a_2,a_1)x$, which is \eqref{eq:reduced}. The quantities $\Phi$ and $\delta$ enter $M$ only through $(\Phi-\delta)I$, and $\frac{d}{dt}\det(M+tI)\big|_{t=0}=\tr M$ for $2\times2$ matrices.
\end{proof}

\subsection{Special cases}

We now specialize to the two threshold distributions of main interest and to the limiting cases.

\begin{corollary}\label{cor:cases}
\begin{enumerate}
\item[(a)] \emph{Constant threshold $G\equiv g_0$.} Then $g=g_0p$, and
\[
M(\gamma)=\begin{pmatrix}
\frac{\sigma_1^2}{2}\gamma^2-c\gamma-\lambda_1-\delta+\dfrac{\lambda_1\beta g_0}{\beta-\gamma} & \dfrac{\lambda_1\beta(1-g_0)}{\beta-\gamma}\\[2.2ex]
\dfrac{\lambda_2\beta g_0}{\beta-\gamma} & \frac{\sigma_1^2}{2}\gamma^2-c\gamma-\lambda_2-\delta+\dfrac{\lambda_2\beta(1-g_0)}{\beta-\gamma}
\end{pmatrix}.
\]
\item[(b)] \emph{Exponential threshold $G(y)=1-e^{-\kappa y}$.} Then
\[
g(\gamma)=\frac{\beta\kappa}{(\beta-\gamma)(\beta+\kappa-\gamma)},\qquad p(\gamma)-g(\gamma)=\frac{\beta}{\beta+\kappa-\gamma}.
\]
\item[(c)] \emph{No switching.} If $g_0=0$ (always regime $2$), then $d=a_1\big(a_2+\lambda_2p\big)$. If $g_0=1$ (always regime $1$), then $d=a_2\big(a_1+\lambda_1p\big)$. In both cases the relevant factor,
\[
\tfrac12\sigma_1^2\gamma^2-c\gamma-\lambda-\delta+\frac{\lambda\beta}{\beta-\gamma},
\]
is the classical Lundberg function of the diffusion-perturbed compound Poisson model \cite{GerberLandry1998}, written in terms of $\gamma=-s$. The same holds for the exponential threshold in the limits $\kappa\to0$ and $\kappa\to\infty$.
\item[(d)] \emph{Independence.} If $\lambda_1=\lambda_2=\lambda$, then $d=a\,(a+\lambda p)$ with $a=\Phi-\lambda-\delta$, and again the second factor is the classical Lundberg function.
\end{enumerate}
\end{corollary}

\begin{proof}
All statements follow by substituting the corresponding $g$ into \eqref{eq:reduced}. For (b), the integral in \eqref{eq:pg} splits into two elementary exponential integrals.
\end{proof}

\section{The adjustment coefficient and Lundberg asymptotics}\label{sec:adjustment}

This section shows that the generalized Lundberg equation has a positive root, the adjustment coefficient, and determines its sensitivities. It then shows that this root governs the actual decay of the Gerber--Shiu functions.

\subsection{Existence and signs}

We first establish existence and the signs that drive everything later. Let $\pi=g(0)=\mathbb{P}(Y<X)$.

\begin{proposition}\label{prop:roots}
Assume $\sigma_2=a=0$, $\lambda_1\ne\lambda_2$, $0<\pi<1$, $\E[Y]<\infty$ and $A_0\beta^2-B_0\beta<\lambda_2+\delta$. Then:
\begin{enumerate}
\item[(i)] $d(0)=\delta^2+\delta\big[\lambda_1(1-\pi)+\lambda_2\pi\big]>0$.
\item[(ii)] $d$ has a root in $(0,\beta)$. Let $\gamma_0$ be the smallest one (the adjustment coefficient).
\item[(iii)] $M_{11}<0$ and $M_{22}<0$ on $[0,\gamma_0]$. Hence $T(\gamma_0)<0$ and $d'(\gamma_0)\le0$, with strict inequality if $\gamma_0$ is simple.
\item[(iv)] $d$ also has a root in $(\gamma_*,0)$, where $\Phi(\gamma_*)=\min(\lambda_1,\lambda_2)+\delta$.
\end{enumerate}
\end{proposition}

\begin{proof}
(i) At $\gamma=0$: $\Phi=0$, $p=1$, $g=\pi$, $a_j=-(\lambda_j+\delta)$. By \eqref{eq:reduced},
\[
d(0)=(\lambda_1+\delta)(\lambda_2+\delta)-\lambda_1\pi(\lambda_2+\delta)-\lambda_2(1-\pi)(\lambda_1+\delta)=\delta^2+\delta\big[\lambda_1(1-\pi)+\lambda_2\pi\big]>0 .
\]
(ii) As $\gamma\uparrow\beta$, $p-g=\beta\int_0^\infty e^{-(\beta-\gamma)y}(1-G)\,dy\to\beta\E[Y]<\infty$ by monotone convergence, while $g\to\infty$. Hence $d=\lambda_1g\,a_2+O(1)$, with $a_2(\beta)=A_0\beta^2-B_0\beta-\lambda_2-\delta<0$. So $d\to-\infty$, and the intermediate value theorem gives a root in $(0,\beta)$.

(iii) On $[0,\beta)$, $M_{12}=\lambda_1(p-g)>0$ and $M_{21}=\lambda_2g>0$. Since $d>0$ on $[0,\gamma_0)$, $M_{11}M_{22}=d+M_{12}M_{21}>0$ on $[0,\gamma_0]$. So the diagonal entries do not vanish there and keep their signs at $0$: $M_{11}(0)=-\lambda_1(1-\pi)-\delta<0$ and $M_{22}(0)=-\lambda_2\pi-\delta<0$. Finally, $d$ is positive on $[0,\gamma_0)$ and vanishes at $\gamma_0$, so $d'(\gamma_0)\le0$.

(iv) Say $\lambda_2<\lambda_1$ (the other case is symmetric). At $\gamma_*<0$ with $\Phi(\gamma_*)=\lambda_2+\delta$ we have $a_2=0$ and $a_1=\lambda_2-\lambda_1<0$, so $d(\gamma_*)=\lambda_2(p-g)a_1<0$. With (i), there is a root in $(\gamma_*,0)$.
\end{proof}

\subsection{Sensitivities}

Next we compute how the adjustment coefficient responds to the model parameters. Let $\theta$ denote any model parameter, and write $\pd_\theta d$ for the partial derivative of $d$ at fixed $\gamma$.

\begin{proposition}\label{prop:sens}
Let $\gamma_0$ be a simple root. Then $\gamma_0$ is a smooth function of the parameters, with
\[
\frac{\pd\gamma_0}{\pd\theta}=-\frac{\pd_\theta d(\gamma_0)}{d'(\gamma_0)} .
\]
In particular, under the hypotheses of Proposition~\ref{prop:roots},
\[
\frac{\pd\gamma_0}{\pd c}=\frac{\gamma_0T}{d'}>0,\qquad \frac{\pd\gamma_0}{\pd(\sigma_1^2/2)}=-\frac{\gamma_0^2T}{d'}<0,\qquad \frac{\pd\gamma_0}{\pd\delta}=\frac{T}{d'}>0,
\]
with $T=T(\gamma_0)$ and $d'=d'(\gamma_0)$.
\end{proposition}

\begin{proof}
The first formula is the implicit function theorem applied to $d(\gamma;\theta)=0$. By Lemma~\ref{lem:reduced}, $\pd_cd=\pd_\Phi d\cdot\pd_c\Phi=-\gamma T$, $\pd_{A_0}d=\gamma^2T$ and $\pd_\delta d=-T$. The signs follow from $T<0$ and $d'<0$ (Proposition~\ref{prop:roots}(iii)).
\end{proof}

A higher premium therefore increases the adjustment coefficient, and a higher volatility decreases it, as expected. The effect of the dependence parameters $\lambda_1,\lambda_2,\kappa$ is given by the same formula and is quantified in Section~\ref{sec:numerics}.

\subsection{Cram\'er--Lundberg asymptotics for the exponential threshold}

Finally, we connect the invariant solutions to the actual Gerber--Shiu functions. For the exponential threshold the half-line system \eqref{eq:gs} is equivalent to a linear ODE system. Put
\begin{gather*}
Y=(m_1,m_1',m_2,m_2',J_1,J_2,J_3)^{\mathsf T},\\
J_1=\int_0^ue^{-\beta(u-x)}m_1\,dx,\qquad J_2=\int_0^ue^{-(\beta+\kappa)(u-x)}m_1\,dx,\qquad J_3=\int_0^ue^{-(\beta+\kappa)(u-x)}m_2\,dx .
\end{gather*}
Then the convolution in \eqref{eq:gs} equals $\beta(J_1-J_2+J_3)$, and $J_1'=m_1-\beta J_1$, $J_2'=m_1-(\beta+\kappa)J_2$, $J_3'=m_2-(\beta+\kappa)J_3$. Hence \eqref{eq:gs} with $\sigma_2=a=0$ reads
\begin{equation}\label{eq:ode7}
Y'=CY+e^{-\beta u}s,\qquad Y_1(0)=Y_3(0)=w(0),\quad Y_5(0)=Y_6(0)=Y_7(0)=0,
\end{equation}
with a constant matrix $C$ and a constant vector $s$, together with boundedness on $(0,\infty)$.

\begin{theorem}\label{thm:asymptotics}
Let $G(y)=1-e^{-\kappa y}$, $\sigma_2=a=0$, $\lambda_1\neq\lambda_2$, and let $\gamma_0\in(0,\beta)$ be a simple root of $d$ with $\Phi(\gamma_0)\notin\{\lambda_1+\delta,\lambda_2+\delta\}$. Assume:
\begin{itemize}
\item $C$ has no eigenvalue on the imaginary axis and exactly five eigenvalues with negative real part, all simple;
\item the $5\times5$ matrix formed by the components $1,3,5,6,7$ of the corresponding eigenvectors is invertible;
\item $-\beta$ is not an eigenvalue of $C$;
\item $-\gamma_0$ is the stable eigenvalue of largest real part, every other stable eigenvalue has real part $<-\gamma_0-\Delta$ for some $\Delta>0$, and the coefficient of $e^{-\gamma_0u}$ in the solution is non-zero.
\end{itemize}
Then there is $a_0\neq0$ such that
\begin{equation}\label{eq:asymp}
m_j(u)=a_0\,v_{0,j}\,e^{-\gamma_0u}\big(1+O(e^{-\Delta'u})\big),\qquad u\to\infty,\quad j=1,2,
\end{equation}
with $\Delta'=\min(\Delta,\beta-\gamma_0)$ and $v_0=\big(\lambda_1/(\lambda_1+\delta-\Phi(\gamma_0)),\ \lambda_2/(\lambda_2+\delta-\Phi(\gamma_0))\big)$, a null vector of $M(\gamma_0)$. In particular, $-m_j'(u)/m_j(u)\to\gamma_0$ and
\begin{equation}\label{eq:ratio}
\lim_{u\to\infty}\frac{m_2(u)}{m_1(u)}=\frac{\lambda_2\big(\lambda_1+\delta-\Phi(\gamma_0)\big)}{\lambda_1\big(\lambda_2+\delta-\Phi(\gamma_0)\big)} .
\end{equation}
\end{theorem}

\begin{proof}
\emph{Step 1: the solution is an exponential polynomial.} Since $-\beta$ is not an eigenvalue, $Y_p=e^{-\beta u}(-\beta I-C)^{-1}s$ is a particular solution. Boundedness excludes the unstable modes, so $Y=Y_p+\sum_{k=1}^5c_ke^{\mu_ku}w_k$ with stable eigenpairs $(\mu_k,w_k)$. The five boundary conditions determine the $c_k$ uniquely by the invertibility assumption. The Gerber--Shiu functions are bounded and solve \eqref{eq:ode7}, so they are given by this formula.

\emph{Step 2: the slowest mode is an invariant solution.} Let $e^{-\gamma_0u}P$ be the term of frequency $\gamma_0$ in $m=(m_1,m_2)$; here $P$ is constant because $-\gamma_0$ is a simple eigenvalue. Insert $m$, extended by the exterior data, into the far-field equation $\G_0m=0$ on $(0,\infty)$ (Proposition~\ref{prop:exterior}). For a term $e^{-\nu u}$ on $(0,\infty)$ extended by zero, and $r\in\{\beta,\beta+\kappa\}$ with $r\neq\nu$,
\[
\int_0^ue^{-ry}e^{-\nu(u-y)}dy=\frac{e^{-\nu u}-e^{-ru}}{r-\nu}.
\]
This produces the full-line contribution at frequency $\nu$ plus a term at frequency $r$. The exterior data contribute only at frequency $\beta$. Since $\gamma_0\notin\{\beta,\beta+\kappa\}$ and exponentials with distinct exponents are linearly independent, the frequency-$\gamma_0$ component of the equation is $\G_0[e^{-\gamma_0u}P]=e^{-\gamma_0u}M(\gamma_0)P=0$ (Theorem~\ref{thm:lundberg}). Hence $P\in\ker M(\gamma_0)$. Since $M(\gamma_0)$ has rank one and $\Phi(\gamma_0)\neq\lambda_j+\delta$, the kernel is spanned by $v_0$: row $j$ of $M(\gamma_0)k=0$ gives $k_j=\lambda_j s/(\lambda_j+\delta-\Phi)$ with $s=g k_1+(p-g)k_2\neq0$. Thus $P=a_0v_0$ with $a_0\neq0$ by assumption.

\emph{Step 3: conclusion.} All other exponents are $\beta$ and $-\mu_k$ with $\operatorname{Re}(-\mu_k)>\gamma_0+\Delta$, which gives \eqref{eq:asymp}. Since both components of $v_0$ are non-zero, \eqref{eq:ratio} follows. The same expansion holds for derivatives, so $-m_j'/m_j\to\gamma_0$.
\end{proof}

\begin{remark}
Equation \eqref{eq:ratio} gives a simple, interpretable consequence of the dependence. Since $\Phi(\gamma_0)<0$ and $\lambda_1>\lambda_2$, the limiting ratio is below one: a portfolio that has just entered the low-risk regime carries asymptotically a fixed fraction of the risk of one in the high-risk regime, and the fraction is explicit.
\end{remark}

\section{Numerical illustration}\label{sec:numerics}

This section quantifies the results for a representative portfolio and isolates the effect of the dependence mechanism on the adjustment coefficient.

\subsection{Parameters}

We use the base parameters
\[
\beta=1,\quad \lambda_1=1.2,\quad \lambda_2=0.8,\quad \delta=0.05,\quad \sigma_1^2=0.2\ (A_0=0.1),\quad c=1.5 .
\]
For the exponential threshold we take $\kappa=0.5$, so $\pi=\kappa/(\beta+\kappa)=1/3$. The condition $A_0\beta^2-B_0\beta=-1.4<\lambda_2+\delta=0.85$ of Proposition~\ref{prop:roots} holds. The roots of $d$ in $\{\operatorname{Re}\gamma<\beta\}$ are $-0.747339$, $-0.073954$ and $\gamma_0=0.396621$, all simple. At $\gamma_0$ we have $T=-1.632124$ and $d'=-2.185350$. The eigenvalues of $C$ in \eqref{eq:ode7} are $-15.790$, $-15.582$, $-1.5528$, $-1.5$, $-0.39662$, $0.07395$ and $0.74734$, and the hypotheses of Theorem~\ref{thm:asymptotics} hold with $\Delta'=\beta-\gamma_0$.

\subsection{The effect of dependence}

We now isolate how the dependence mechanism affects the adjustment coefficient. Table~\ref{tab:dependence} compares three models with the same proportion $\pi$ of switches to the high-risk regime:
\begin{itemize}
\item \emph{random switching}, with constant threshold $G\equiv g_0=\pi$, where the next regime is chosen independently of the claim;
\item \emph{claim-driven switching}, with exponential threshold and $\kappa=\beta\pi/(1-\pi)$, where large claims tend to trigger the high-risk regime;
\item the \emph{independent} classical model whose claim intensity equals the long-run mean intensity $\bar\lambda=\big(\pi/\lambda_1+(1-\pi)/\lambda_2\big)^{-1}$ of the dependent model.
\end{itemize}
Figure~\ref{fig:main}(a) shows the same comparison as a function of $\pi$.

\begin{table}[ht]
\centering
\caption{Adjustment coefficient $\gamma_0$ under random and claim-driven switching, compared with the independent model having the same mean claim intensity $\bar\lambda$.}\label{tab:dependence}
\renewcommand{\arraystretch}{1.2}
\begin{tabular}{@{}lrrrrr@{}}
\toprule
threshold & parameter & $\pi$ & $\gamma_0$ & $\bar\lambda$ & $\gamma_0$, independent\\
\midrule
constant & $g_0=0.25$ & 0.250 & 0.437619 & 0.8727 & 0.443546\\
constant & $g_0=0.50$ & 0.500 & 0.386859 & 0.9600 & 0.395100\\
constant & $g_0=0.75$ & 0.750 & 0.332649 & 1.0667 & 0.338958\\
exponential & $\kappa=0.1$ & 0.091 & 0.456037 & 0.8250 & 0.470851\\
exponential & $\kappa=0.5$ & 0.333 & 0.396621 & 0.9000 & 0.428191\\
exponential & $\kappa=1$ & 0.500 & 0.363247 & 0.9600 & 0.395100\\
exponential & $\kappa=2$ & 0.667 & 0.332689 & 1.0286 & 0.358579\\
exponential & $\kappa=5$ & 0.833 & 0.303514 & 1.1077 & 0.318420\\
\bottomrule
\end{tabular}
\\[4pt]
\small Limits: $\gamma_0=0.485359$ ($\pi=0$, classical with $\lambda_2$) and $\gamma_0=0.274719$ ($\pi=1$, classical with $\lambda_1$).
\end{table}

Two effects are visible. First, both dependent models have a smaller adjustment coefficient than the independent model with the same mean claim intensity, so replacing the dependent model by its ``average'' independent counterpart underestimates the risk. Second, the reduction is much larger when switching is driven by claim size. At $\pi=1/2$ the reduction is $2.1\%$ under random switching ($0.3869$ versus $0.3951$) but $8.1\%$ under claim-driven switching ($0.3632$ versus $0.3951$). The reason is that, under claim-driven switching, large claims are followed by a higher claim intensity, so large losses tend to cluster. This clustering is invisible both in the mean intensity and in models with random regime changes.

\begin{figure}[ht]
\centering
\includegraphics[width=\textwidth]{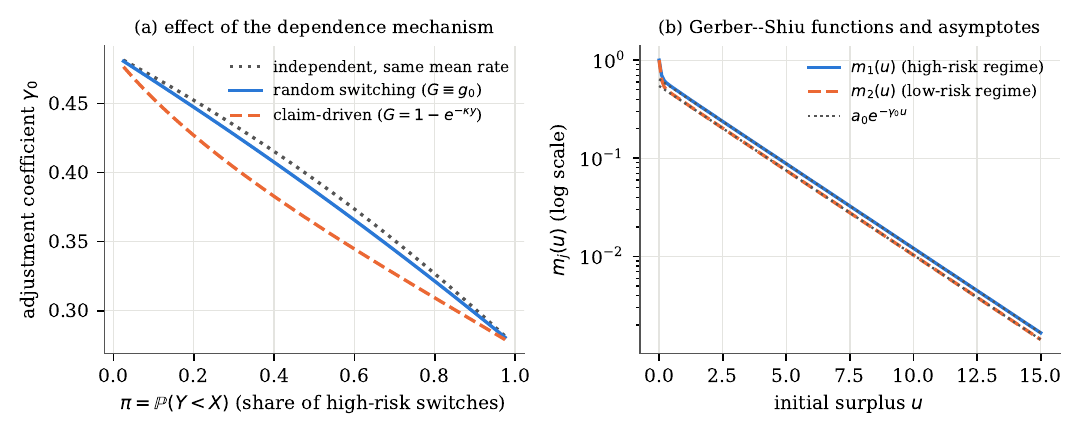}
\caption{(a) Adjustment coefficient as a function of the share $\pi$ of high-risk switches, for random switching (solid), claim-driven switching (dashed) and the independent model with the same mean intensity (dotted). (b) Gerber--Shiu functions $m_1,m_2$ for $w\equiv1$ (Laplace transform of the time of ruin), with exponential threshold $\kappa=0.5$, on a logarithmic scale, together with the asymptotes $a_0v_{0,j}e^{-\gamma_0u}$ of Theorem~\ref{thm:asymptotics} (dotted).}\label{fig:main}
\end{figure}

\subsection{Sensitivities}

Next we report the sensitivities of Proposition~\ref{prop:sens} for the base case with exponential threshold $\kappa=0.5$:
\begin{gather*}
\frac{\pd\gamma_0}{\pd\lambda_1}=-0.169,\qquad \frac{\pd\gamma_0}{\pd\lambda_2}=-0.334,\qquad \frac{\pd\gamma_0}{\pd\kappa}=-0.095,\qquad \frac{\pd\gamma_0}{\pd\beta}=0.865,\\
\frac{\pd\gamma_0}{\pd c}=0.296,\qquad \frac{\pd\gamma_0}{\pd A_0}=-0.117,\qquad \frac{\pd\gamma_0}{\pd\delta}=0.747 .
\end{gather*}
The values for $c$, $A_0$ and $\delta$ agree with the closed forms $\gamma_0T/d'$, $-\gamma_0^2T/d'$ and $T/d'$. The adjustment coefficient is twice as sensitive to the low-risk intensity $\lambda_2$ as to the high-risk intensity $\lambda_1$, because two thirds of the inter-claim intervals are spent in the low-risk regime when $\pi=1/3$. Increasing $\kappa$ shifts switches toward the high-risk regime and lowers $\gamma_0$.

\subsection{Gerber--Shiu functions}

Finally, we illustrate Theorem~\ref{thm:asymptotics}. Solving \eqref{eq:ode7} exactly for $w\equiv1$ gives the Laplace transform of the time of ruin in each regime. Table~\ref{tab:gs} lists selected values; Figure~\ref{fig:main}(b) shows the functions with their asymptotes, where $a_0=0.63920$. The ratio $m_2/m_1$ converges rapidly to the limit $0.853251$ predicted by \eqref{eq:ratio}.

\begin{table}[ht]
\centering
\caption{Conditional Gerber--Shiu functions ($w\equiv1$, $\delta=0.05$, $\kappa=0.5$) and their ratio.}\label{tab:gs}
\renewcommand{\arraystretch}{1.2}
\begin{tabular}{@{}rrrr@{}}
\toprule
$u$ & $m_1(u)$ & $m_2(u)$ & $m_2(u)/m_1(u)$\\
\midrule
0 & 1.000000 & 1.000000 & 1.000000\\
1 & 0.434184 & 0.370059 & 0.852310\\
2 & 0.290062 & 0.247409 & 0.852953\\
5 & 0.087989 & 0.075076 & 0.853242\\
10 & 0.012110 & 0.010333 & 0.853251\\
\midrule
$\infty$ & & & 0.853251 (eq.~\eqref{eq:ratio})\\
\bottomrule
\end{tabular}
\end{table}

\section{Discussion and conclusion}\label{sec:conclusion}

We conclude by summarizing what the symmetry analysis reveals about the model and what remains open.

We have carried out a Lie point symmetry analysis of the coupled integro-differential system for the conditional Gerber--Shiu functions of a dependent jump-diffusion risk model. Two features were essential. First, the system is the restriction of a far-field operator with exterior data, which separates the trivial breaking of symmetry by the boundary from the structure of the bulk dynamics. Second, the nonlocal convolution term contributes its own determining equation, and the $m$-coefficient of the local part cannot be omitted: without these, the remaining equations admit spurious candidates (Remark~\ref{rem:partial}).

The resulting group classification (Theorem~\ref{thm:classification}) is sharp. A non-trivial symmetry beyond scaling and superposition exists if and only if the volatility and the premium are constant, and it is then the translation in the surplus variable. Surplus-dependent volatility or premium growth, as caused for instance by investment of the surplus, destroys all non-trivial point symmetries (Corollary~\ref{cor:breaking}). The dependence between claim sizes and inter-claim times rigidifies the algebra: it forces equal scalings in the two regimes and equal multipliers, and it rules out symmetries with non-constant $\xi$.

In the symmetric case, the abelian algebra $\langle\pd_u,Y_0\rangle$ has the one-parameter family of subalgebras $\langle\pd_u-\gamma Y_0\rangle$, whose invariant solutions are exponentials. They exist exactly when the generalized coupled Lundberg equation holds. The equation is valid for every threshold distribution, depends on volatility, premium and discount only through $\Phi-\delta$, and reduces to the classical Lundberg equation in the limits of no switching and of independence. The smallest positive root, the adjustment coefficient, exists under a simple condition. It has explicit sensitivities and governs the exponential decay of the Gerber--Shiu functions, whose regime ratio converges to an explicit limit. Numerically, claim-driven dependence reduces the adjustment coefficient substantially below that of the independent model with the same mean claim intensity, an effect invisible to models based on average intensities or random regime changes.

Several extensions are natural. The first is the symmetry structure of the independent case (Remark~\ref{rem:indep}) and of models with more than two regimes. The second is the analysis of generators that are not fiber-preserving or not linear. The third is the regime where the coefficients depend weakly on the surplus, in which the translation symmetry is broken but its consequences for the decay rate may persist approximately; approximate symmetry methods are tool there. 
\section*{Declarations}

\noindent\textbf{Data availability.} No data were used in this study. 

\smallskip
\noindent\textbf{Conflict of interest.} The authors declare that they have no conflict of interest.


\begin{thebibliography}{99}

\bibitem{AlbrecherBoxma2004} H.~Albrecher, O.~J.~Boxma, A ruin model with dependence between claim sizes and claim intervals, Insurance Math.\ Econom.\ \textbf{35} (2004), 245--254.

\bibitem{AlbrecherTeugels2006} H.~Albrecher, J.~L.~Teugels, Exponential behavior in the presence of dependence in risk theory, J.\ Appl.\ Probab.\ \textbf{43} (2006), 257--273.

\bibitem{AsmussenAlbrecher2010} S.~Asmussen, H.~Albrecher, Ruin Probabilities, 2nd ed., World Scientific, Singapore, 2010.

\bibitem{AzadEtAl2010} H.~Azad, M.~T.~Mustafa, F.~M.~Arif, Analytic solutions of initial-boundary-value problems of transient conduction using symmetries, Appl.\ Math.\ Comput.\ \textbf{215} (2010), 4132--4140.

\bibitem{BadescuEtAl2009} A.~L.~Badescu, E.~C.~K.~Cheung, D.~Landriault, Dependent risk models with bivariate phase-type distributions, J.\ Appl.\ Probab.\ \textbf{46} (2009), 113--131.

\bibitem{BlumanKumei1989} G.~W.~Bluman, S.~Kumei, Symmetries and Differential Equations, Springer, New York, 1989.

\bibitem{BoudreaultEtAl2006} M.~Boudreault, H.~Cossette, D.~Landriault, E.~Marceau, On a risk model with dependence between interclaim arrivals and claim sizes, Scand.\ Actuar.\ J.\ \textbf{2006} (2006), 265--285.

\bibitem{CheungEtAl2010} E.~C.~K.~Cheung, D.~Landriault, G.~E.~Willmot, J.-K.~Woo, Structural properties of Gerber--Shiu functions in dependent Sparre Andersen models, Insurance Math.\ Econom.\ \textbf{46} (2010), 117--126.

\bibitem{Cramer1930} H.~Cram\'er, On the Mathematical Theory of Risk, Skandia Jubilee Volume, Stockholm, 1930.

\bibitem{DufresneGerber1991} F.~Dufresne, H.~U.~Gerber, Risk theory for the compound Poisson process that is perturbed by diffusion, Insurance Math.\ Econom.\ \textbf{10} (1991), 51--59.

\bibitem{GaetaQuintero1999} G.~Gaeta, N.~R.~Quintero, Lie-point symmetries and stochastic differential equations, J.\ Phys.\ A: Math.\ Gen.\ \textbf{32} (1999), 8485--8505.

\bibitem{GaetaSpadaro2017} G.~Gaeta, F.~Spadaro, Random Lie-point symmetries of stochastic differential equations, J.\ Math.\ Phys.\ \textbf{58} (2017), 053503.

\bibitem{Gao2016} H.~Gao, Integro-differential equations for a jump-diffusion risk process with dependence between claim sizes and claim intervals, J.\ Appl.\ Math.\ Phys.\ \textbf{4} (2016), 2061--2068.

\bibitem{GerberLandry1998} H.~U.~Gerber, B.~Landry, On the discounted penalty at ruin in a jump-diffusion and the perpetual put option, Insurance Math.\ Econom.\ \textbf{22} (1998), 263--276.

\bibitem{GerberShiu1998} H.~U.~Gerber, E.~S.~W.~Shiu, On the time value of ruin, N.\ Am.\ Actuar.\ J.\ \textbf{2} (1998), 48--72.

\bibitem{GrigorievEtAl2010} Y.~N.~Grigoriev, N.~H.~Ibragimov, V.~F.~Kovalev, S.~V.~Meleshko, Symmetries of Integro-Differential Equations, Lecture Notes in Physics 806, Springer, Dordrecht, 2010.

\bibitem{Ibragimov1994} N.~H.~Ibragimov (ed.), CRC Handbook of Lie Group Analysis of Differential Equations, Vols.~1--3, CRC Press, Boca Raton, 1994--1996.

\bibitem{KumarEtAl2022} S.~Kumar, A.~Kumar, M.~Inc, H.~Alotaibi, M.~A.~Abdou, A.~Akgul, An investigation of $(2+1)$-dimensional asymmetric Nizhnik--Novikov--Veselov system: Lie symmetry reductions, invariant solutions, dynamical behaviors and conservation laws, Results Phys.\ \textbf{43} (2022), 106034.

\bibitem{Lundberg1903} F.~Lundberg, Approximerad framst\"allning af sannolikhetsfunktionen. \AA terf\"ors\"akring af kollektivrisker, Almqvist \& Wiksell, Uppsala, 1903.

\bibitem{Miller1977} W.~Miller, Jr., Symmetry and Separation of Variables, Addison-Wesley, Reading, MA, 1977.

\bibitem{Nass2019} A.~M.~Nass, Lie symmetry analysis and exact solutions of ordinary differential equation with neutral delay, Appl.\ Math.\ Comput.\ \textbf{347} (2019), 370--380.

\bibitem{Nass2022} A.~M.~Nass, Lie symmetries of autonomous scalar first-order It\^o stochastic delay differential equations, J.\ Theoret.\ Probab.\ \textbf{35} (2022), 1939--1951.

\bibitem{NassFredericks2017} A.~M.~Nass, E.~Fredericks, W-symmetries of jump-diffusion It\^o stochastic differential equations, Nonlinear Dynam.\ \textbf{90} (2017), 2869--2877.

\bibitem{Olver1993} P.~J.~Olver, Applications of Lie Groups to Differential Equations, 2nd ed., Springer, New York, 1993.

\bibitem{Ovsiannikov1982} L.~V.~Ovsiannikov, Group Analysis of Differential Equations, Academic Press, New York, 1982.

\bibitem{TsaiWillmot2002} C.~C.-L.~Tsai, G.~E.~Willmot, A generalized defective renewal equation for the surplus process perturbed by diffusion, Insurance Math.\ Econom.\ \textbf{30} (2002), 51--66.

\bibitem{Unal2003} G.~\"Unal, Symmetries of It\^o and Stratonovich dynamical systems and their conserved quantities, Nonlinear Dynam.\ \textbf{32} (2003), 417--426.

\bibitem{YuFeng2025a} J.~Yu, Y.~Feng, Time fractional $(2+1)$-dimensional Heisenberg ferromagnetic spin chain equation: its Lie symmetries, exact solutions and conservation laws, Commun.\ Theor.\ Phys.\ \textbf{77} (2025), 055002.

\bibitem{YuFeng2025b} J.~Yu, Y.~Feng, Lie symmetries, power series solutions and conservation laws of $(2+1)$-dimensional nonlinear fractional Fokas system, Quantum Stud.\ Math.\ Found.\ \textbf{12} (2025), 14.

\bibitem{ZhangYang2011} Z.~Zhang, H.~Yang, Gerber--Shiu analysis in a perturbed risk model with dependence between claim sizes and interclaim times, J.\ Comput.\ Appl.\ Math.\ \textbf{235} (2011), 1189--1204.

\end{thebibliography}
\end{document}